\documentclass[11pt,openany,leqno]{article}
\usepackage{amsmath,amsthm,amsfonts,amssymb,amscd,url}
\usepackage{graphics}
\usepackage[latin1]{inputenc}
\usepackage{hyperref}
\usepackage{epsfig}

\input{xy} \xyoption{all}

\usepackage{enumerate}
\usepackage{hyperref}
\usepackage{epsfig} 
\input{xy} \xyoption{all}

\begin{document}

\def\Diff{\text{Diff}}
\def\Max{\text{max}}
\def\P{\mathbb P}
\def\R{\mathbb R}
\def\T{\mathbb{T}}
\def\N{\mathbb N}
\def\Z{\mathbb Z}
\def\C{\mathbb C}
\def\D{\mathbb D}
\def\a{{\underline a}}
\def\b{{\underline b}}
\def\n{{\underline n}}
\def\Log{\text{log}}
\def\loc{\text{loc}}
\def\inta{\text{int }}
\def\det{\text{det}}
\def\exp{\text{exp}}
\def\Re{\text{Re}}
\def\lip{\text{Lip}}
\def\leb{\text{Leb}}
\def\dom{\text{Dom}}
\def\diam{\text{diam}\:}
\def\supp{\text{supp}\:}
\newcommand{\ovfork}{{\overline{\pitchfork}}}
\newcommand{\ovforki}{{\overline{\pitchfork}_{I}}}
\newcommand{\Tfork}{{\cap\!\!\!\!^\mathrm{T}}}
\newcommand{\whforki}{{\widehat{\pitchfork}_{I}}}
\newcommand{\marginal}[1]{\marginpar{{\scriptsize {#1}}}}
\def é{{\' e}}
\def\sR{{\mathfrak R}}
\def\sM{{\mathfrak M}}
\def\sA{{\mathfrak A}}
\def\sB{{\mathfrak B}}
\def\sY{{\mathfrak Y}}
\def\sE{{\mathfrak E}}
\def\sP{{\mathfrak P}}
\def\sG{{\mathfrak G}}
\def\sa{{\mathfrak a}}
\def\sb{{\mathfrak b}}
\def\sc{{\mathfrak c}}
\def\se{{\mathfrak e}}
\def\sg{{\mathfrak g}}
\def\sd{{\mathfrak d}}
\def\sr{{\mathfrak {r}}}
\def\ss{{\mathfrak {s}}}
\def\sD{{\mathfrak {p}}}
\def\sp{{\mathfrak {p}}}
\def\arr{\overleftarrow}
\def\u{\underline}

\newtheorem{prop}{Proposition} [section]
\newtheorem{thm}[prop] {Theorem}
\newtheorem{conj}[prop] {Conjecture}
\newtheorem{defi}[prop] {Definition}
\newtheorem{lemm}[prop] {Lemma}

\newtheorem{prob}[prop] {Problem}

\newtheorem{sublemm}[prop] {Sub-Lemma}
\newtheorem{cor}[prop]{Corollary}
\newtheorem{theo}{Theorem}
\newtheorem{theoprime}{Theorem}
\newtheorem{Claim}[prop]{Claim}
\newtheorem{fact}[prop]{Fact}

\newtheorem{coro}[theo]{Corollary}
\newtheorem{defprop}[prop]{Definition-Proposition}
\newtheorem{propdef}[prop]{Proposition-Definition}

\newtheorem{question}[prop]{Question}
\newtheorem{conjecture}[prop]{Conjecture}

\theoremstyle{remark}
\newtheorem{exam}[prop]{Example}
\newtheorem{rema}[prop]{Remark}

\renewcommand{\thetheo}{\Alph{theo}}
\renewcommand{\thetheoprime}{\Alph{theo}$'$}
\newtheorem{propfonda}[theo]{\bf Fundamental property of the parablenders}

\title{
Lectures on Structural Stability in Dynamics
}

\author{Pierre Berger\footnote{CNRS-LAGA, Université Paris 13.}}

\date{}

\maketitle

\begin{abstract} 
These lectures  present  results and problems on the characterization of structurally stable dynamics. 
We will shed light those which do not seem to  depend on the regularity class (holomorphic or differentiable).
Furthermore, we will present some links between the problems of structural stability in dynamical systems and in singularity theory. 
\end{abstract}
\tableofcontents

\section*{Introduction}

Structural stability is one of the most basic topics in dynamical systems and contains some of the hardest conjectures.
  Given a class of regularity $\mathcal C$, which can be $C^r$ for $1\le r\le \infty$ or holomorphic, and formed by diffeomorphisms or endomorphisms of a manifold $M$, the problem is to describe the \emph{structurally stable dynamics} for the class $\mathcal C$. We recall that a dynamics $f$ is $\mathcal C$-structurally stable if for any perturbation $\hat f$ of $f$ in the class $\mathcal C$, there exists a homeomorphism $h$ of $M$ so that $h\circ f= f\circ h$.  
Uniform hyperbolicity seems to provide a satisfactory way to describe the structurally stable dynamics. This observation goes back to the Fatou conjecture for quadratic maps of the Riemannian sphere in 1920 and the Smale conjecture for smooth diffeomorphisms in 1970. These conjectures have been deeply studied by many mathematicians and so they are difficult to tackle directly. 

However at the interface of one-dimensional complex dynamics and differentiable dynamics, the field of two-dimensional complex dynamics  grew up recently. It enables to study the structural stability problem thanks to ingredients of both  fields. 
Also the mathematics  developed in the 1970's for the structural stability in dynamics is very similar to the one developed for the structural stability in singularity theory. This led us to combine both in the study of the structurally stable endomorphisms,  
We will review some classical works  in these beautiful fields,  some works more recent,   and we will present  new open problems at these interfaces.

In section \ref{hyp:sec}, we will recall some elementary definitions of uniform hyperbolic theory, and we will detail a few examples of such dynamics.  

In section \ref{secStabimplieshyp} we will state several theorems and conjectures suggesting the hyperbolicity of structurally stable dynamics. In particular we will recall the seminal work of Ma\~n\'e \cite{Ma88} showing this direction in the $C^1$-category. For holomorphic dynamical systems, we will present the work of Dujardin-Lyubich \cite{LD13}  and our work with Dujardin \cite{BD14} generalizing some aspects of 
Ma\~ n\'e-Sad-Sullivan and Lyubich theorems \cite{MSS, 
Ly84} for polynomial  automorphisms of $\C^2$. 

In section \ref{hypimpliesstab}, we will present several results in the directions ``hyperbolicty $\Rightarrow$ stability". In \textsection \ref{hypimpliesstab1}, we will recall classical results, including the structural stability theorems  of Anosov \cite{An67}, Moser\cite{Mo69} and Shub \cite{Shub69}, and the proof of this direction of the $\Omega$-stability conjecture by Smale\cite{Sm68} and Przytycki  \cite{Pr77}. 
Then in \textsection \ref{hypimpliesstab2}, we will sketch the  proof of this direction of  the structural stability theorem by Robbin \cite{Ro71} and Robinson \cite{Ro76} ; and we will relate a few works leading to a generalization of the Przytycki  conjecture \cite{Pr77}, a description of the structurally stable local diffeomorphisms. Finally in \textsection \ref{hypimpliesstab3}, we will recall our conjecture with Rovella \cite{BR13} stating a description of the endomorphisms (with possibly a non-empty critical set) whose inverse limit is structurally stable, and we will state our theorem with Kocsard \cite{BK13} showing one direction of this conjecture. 

In section \ref{sectionLinks}, we will recall several results from singularity theory and we will emphasize on their similarities with those of structural stability.

In section \ref{Sec_SS_endo_w_Singu}, we will present the work
\cite{Be12} which  states sufficient conditions for a smooth map with non-empty critical set to be structurally stable.  The statement involves developments of 
Mather's theorem on Singularity Theory of composed mappings. It suggests the problem of the description of the structurally stable, surface endomorphisms among those which display singularity but satisfy the axiom A.

\medskip
\thanks{These notes were written while I was giving lectures at Montevideo in 2009
and at the Banach Center in 2016. I am very grateful for their hospitality.
}
\section{Uniformly hyperbolic dynamical systems}\label{hyp:sec}
The theory of \emph{uniformly hyperbolic dynamical systems} was constructed in the 
1960's under the dual leadership of Smale in the USA, and Anosov and Sinai in the Soviet Union.\footnote{A few sentences of this section are taken from \cite{BY14}.}
\par
 It encompasses various examples that we shall recall: expanding maps, horseshoes, solenoid maps, Plykin attractors,  Anosov maps, DA, blenders all of which are \emph{basic pieces}. 

\subsection{Uniformly hyperbolic  diffeomorphisms}
 Let $f$ be a $C^1$-diffeomorphism $f$ of a finite dimensional manifold $M$. A compact $f$-invariant subset $\Lambda \subset M$ is 
\emph{uniformly hyperbolic} if the restriction to $\Lambda$ of the tangent bundle $TM$
splits into two continuous invariant subbundles
\[TM|\Lambda = E^s\oplus E^u,\]
$E^s$ being  uniformly contracted and $E^u$ being uniformly expanded: 
 $\exists \lambda<1$, $\exists C>0$,  
\[\|T_xf_{|E^s}^n\|<C\cdot\lambda^n\quad\mathrm{and} \quad \|T_xf_{|E^u}^{-n}\|<C\cdot\lambda^n, \quad \forall x\in \Lambda, \forall n\ge 0.\]

\begin{exam}[Hyperbolic periodic point]
A periodic point at which the differential has no eigenvalue of modulus $1$ is called hyperbolic. It is a sink if all the eigenvalues are  of modulus less than 1, a source if all of them are of modulus greater than 1, and a saddle otherwise. 
\end{exam}

\begin{defi} A \emph{hyperbolic attractor}  is a hyperbolic, transitive compact subset $\Lambda$ such that there exists a neighborhood  $N$ satisfying $\Lambda = \cap_{n \geq 0} f^n(N)$.
\end{defi}
\begin{exam}[Anosov] If the compact hyperbolic set is equal to the whole compact manifold, then the map is called \emph{Anosov}. 
For instance if a map $A\in SL_2(\Z)$ has both eigenvalues of modulus not equal to $1$, then it acts on the torus $\R^2/\Z^2$ as an Anosov diffeomorphism. The following linear map satisfies such a property:  
\[A:= \left[\begin{array}{cc}
2&1\\
1&1\end{array}\right]\; .\]
\end{exam}

\begin{exam}[Smale solenoid]
We consider a perturbation of the map of the filled torus 
$\mathbb T := \{(\theta, z)\R/\Z \times \mathbb C: |z|<1\}$ 
:
\[(\theta, z)\in \mathbb T \mapsto (2\theta, 0)\in  \mathbb T \;,\]
which is a diffeomorphism onto its image. This is the case of the following:  
\[(\theta, z)\in \mathbb T \mapsto (2\theta, \epsilon \cdot z+ 2\epsilon\cdot  \exp(2\pi i \theta )\in  \mathbb T \;.\]
This defines a hyperbolic attractor called the \emph{Smale solenoid}.
\end{exam}
\vspace{1cm}
\begin{exam}[Derivated from Anosov (DA) and Plykin attractor]
We start with a linear Anosov of the 2-torus $\R^2/\Z^2$. It fixes the point 0.  In local coordinates $\phi$ of a neighborhood $V$ of $0$, it has the form for $0<\lambda<1$:
$$(x,y)\mapsto (\lambda x, y/\lambda).$$ 
For every $\epsilon>0$, let $\rho_\epsilon$ be a smooth function 
so that: 
\begin{itemize}
\item it is equal to $x\mapsto \lambda x$ outside of the interval $(-2\epsilon, 2\epsilon)$,
\item $\rho_\epsilon$ displays exactly three fixed point: $-\epsilon$ and $\epsilon$ which are contracting and  $0$ which is expanding. 
\end{itemize}
Let $DA$ be the map of the two torus equal to $A$ outside of $V$, and in the coordinate $\phi$ it has the form:
 $$(x,y)\mapsto (\rho_\epsilon( x), y/\lambda).$$ 

We notice that $0$ is an expanding the fixed point of DA. The complement of its repulsion basin is a hyperbolic attractor. 
   
\begin{figure}[h!]
	\centering
		\includegraphics[width=5cm]{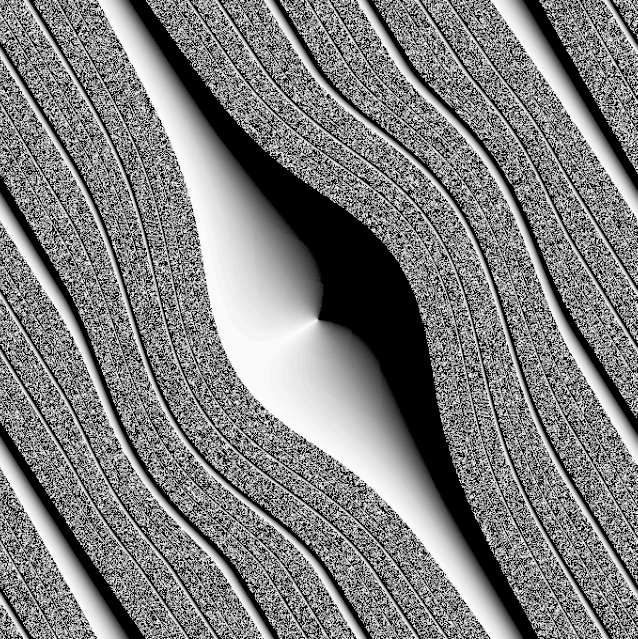}
	\caption{Derivated of Anosov (Credit Y. Coudene \cite{Co06})}
\end{figure}

The DA attractor project to a basic set of a surface attractor, the \emph{Plykin attractor}. 

\begin{figure}[h!]
	\centering
		\includegraphics[width=9cm]{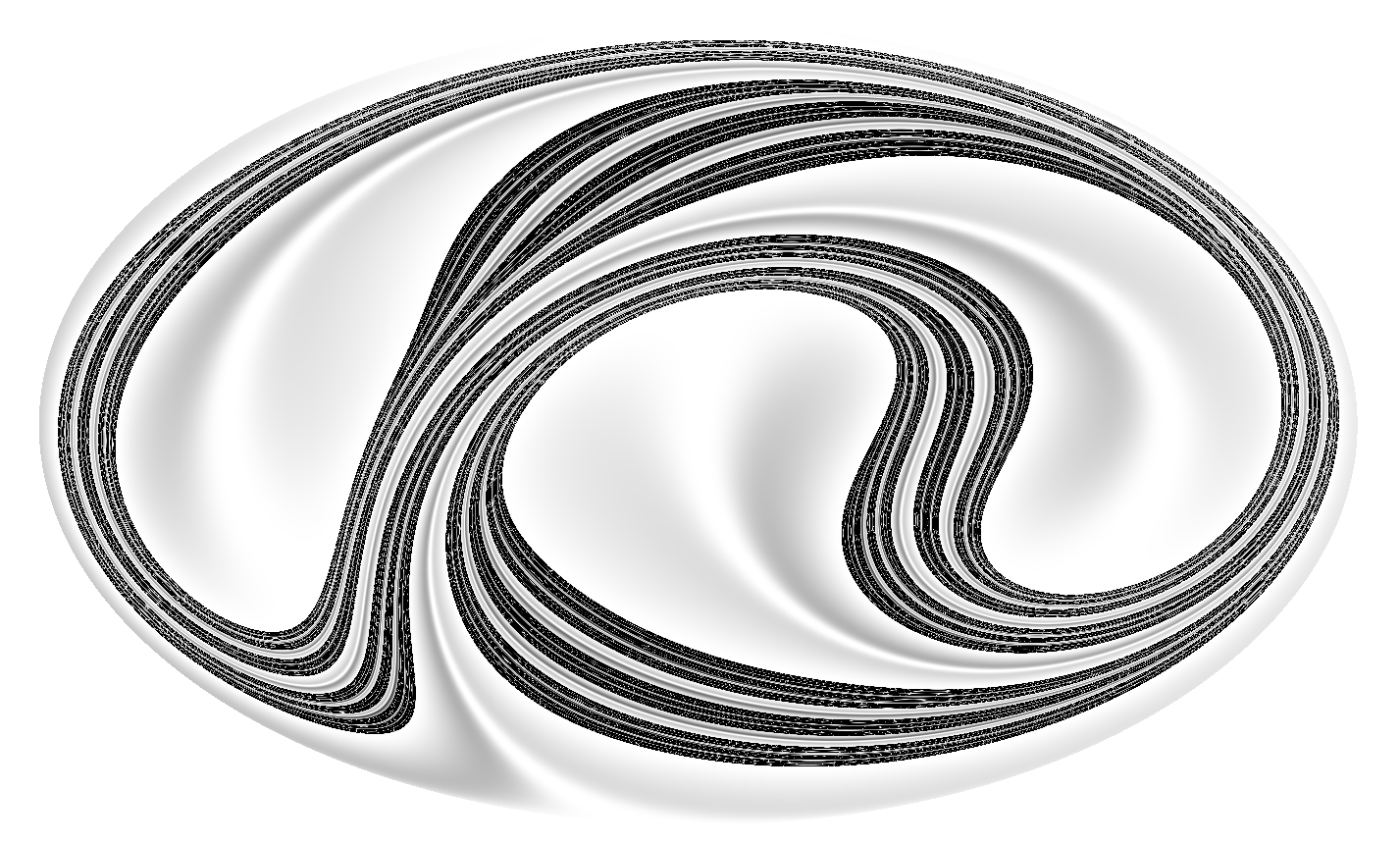}
	\caption{Plykin attractor (Credit S. Crovisier)}
\end{figure}
\end{exam}

Given a hyperbolic compact set $\Lambda$, for every $z\in \Lambda$, the sets 
\[W^s(z)= \{z'\in M:\; \lim_{ n\to+\infty} d(f^n(z),f^n(z'))= 0\},\]
\[W^u(z)= \{z'\in M:\; \lim_{ n\to-\infty} d(f^n(z),f^n(z'))= 0\}\]
are called \emph{the stable and unstable manifolds} of $z$. They are immersed manifolds tangent at $z$ to respectively $E^s(z)$ and $E^u(z)$.

The  \emph{$\epsilon$-local  stable manifold} $ W^s_\epsilon(z)$ of $z$ is the connected component of $z$ in the intersection of $W^s(z)$ with a $\epsilon$-neighborhood  of $z$. The \emph{$\epsilon$-local  unstable manifold} $ W^u_\epsilon(z)$ is defined likewise.

\begin{prop}
For $\epsilon>0$ small enough, the subsets  $ W^s_\epsilon(z)$ and $ W^u_\epsilon(z)$ are $C^r$-embedded manifolds which 
depend continuously on $z$ and tangent at $z$ to respectively  $E^s(z)$ and $E^u(z)$.
\end{prop}
A nice proof of this proposition can be found in \cite{Yoccozintro}.

\begin{defi} 
A \emph{basic set}  is a compact, $f$-invariant, transitive,  uniformly hyperbolic  set $\Lambda$ which is \emph{locally maximal}:
there exists   a neighborhood $N$ of $\Lambda$ such that $\Lambda = \cap_{n\in \Z} f^n(N)$.
\end{defi}

\begin{exam}[Horseshoe] 
A \emph{horseshoe} is a basic set which is a Cantor set. 
For instance take two disjoint sub-intervals $I_+\sqcup I_- \subset [0,1]$, and let $g \colon I_+\sqcup I_-  \to [0,1]$ be a locally affine map which sends each of the intervals $I_\pm$  onto $[0,1]$.  Let $g_+$ be its inverse branch with value in $I_+$ and let $g_-$ be the other inverse branch. 
Let $f$ be a diffeomorphism of the plane whose restriction to $I_\pm \times [0,1]$ is:
\[(x,y)\in  (I_+\sqcup I_-) \times [0,1] \to 
\left\{ \begin{array}{c}
(g(x),g_+(y)) \quad \text{ if }x\in I_+\\
(g(x),g_-(y)) \quad \text{ if }x\in I_-\end{array}\right.\]
\begin{figure}[h!]
		\centering
			\includegraphics[width=7cm]{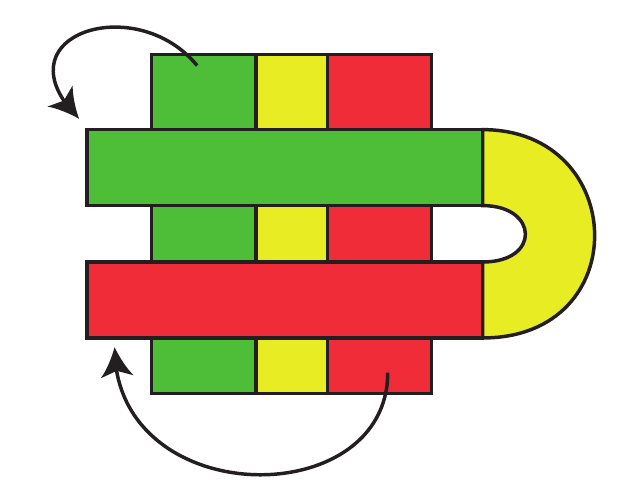}
		\caption{Smale's Horseshoe}
	\end{figure}
\end{exam}

\begin{rema} 
Usually, one defines a basic piece as a hyperbolic set included in the closure of the set of its periodic points. Actually the three following assertion are equivalent for every uniformly hyperbolic, transitive, compact set $K$:
\begin{itemize}
\item  $K$ is locally maximal.
\item $K$ has a structure of local product : for $\epsilon>0$ small enough, and any $x,y\in K$ close enough, the intersection point $W^u_{\epsilon} (x)\cap W^s_{\epsilon} (y)$ belongs to $K$.
\item $K$ is included in the closure of the set of periodic points in $K$: $K= cl(Per(f|K))$. 
\end{itemize}
The equivalence of these conditions is proved in \cite{shubstab78}.
\end{rema}

\begin{defi}[Axiom A] A diffeomorphism whose non-wandering set is a finite union of disjoint  basic sets is called   \emph{axiom A}.
 \end{defi}
 
 \bigskip
 \begin{exam}[Morse-Smale] 
A  Morse-Smale diffeomorphism is a diffeomorphism of a surface so that its non-wandering set consists of finitely many periodic hyperbolic points, and their  stable and unstable manifolds are transverse.
\begin{figure}[h!]
	\centering
		\includegraphics[width=7cm]{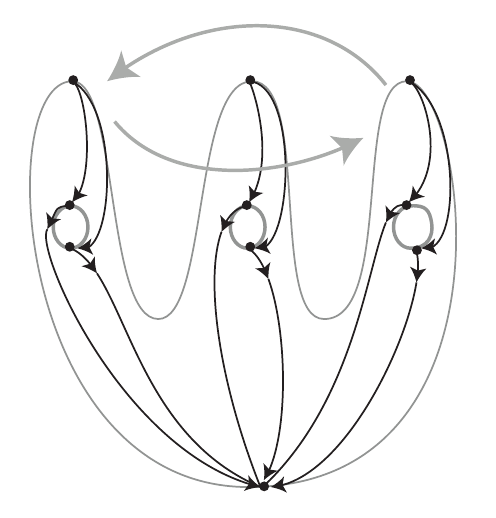}
	\caption{Morse-Smale}
\end{figure}
 \end{exam}

\subsection{Uniformly hyperbolic endomorphisms}

A \emph{$C^r$-endomorphism} of a manifold $M$ is a differentiable map of class $C^r$ of $M$, which is not necessarily injective, nor surjective, and that may possess points at which the differential is not onto (called critical points). The \emph{critical set} is the subset of $M$ of formed by the critical points. 

A \emph{local $C^r$-diffeomorphism} is a $C^r$-endomorphism without critical point.

A compact subset $\Lambda\subset M$ is \emph{invariant} for an endomorphism $f$ of $M$ if $f^{-1} (\Lambda)=\Lambda$. 
A compact subset $\Lambda\subset M$ is \emph{stable} for an endomorphism $f$ of $M$ if $f (\Lambda)=\Lambda$.

An invariant compact set is \emph{hyperbolic} if there exists a subbundle $E^s\subset TM \Lambda$ which is left invariant  and uniformly contracted by $Df$ and so that the action of $Df$ on $TM/E^s$ is uniformly expanding.

\begin{exam}[Expanding map]
Let $f\in End^1(M)$ and  an invariant stable, compact  subset $K$ is \emph{expanded} if there exists $n\ge 1$ s.t., for every $x\in K$,  $D_xf^n$ is invertible and with contracting inverse. When  $K=M$, $f$ is said \emph{expanding}.
\end{exam}   
\begin{exam}[Anosov endomorphism] If a hyperbolic set is equal to the whole manifold, then the endomorphism is called \emph{Anosov}. 
For instance this is the case of the dynamics on the torus $\mathbb R^2/\mathbb Z^2$ induced by a linear maps in $M_2(\Z)$ with eigenvalues of modulus not equal to 1. For instance, it the case of the following for every $n\ge 2$:
\[\left[\begin{array}{cc}
n&1\\
1&1\end{array}\right]\]
\end{exam}    


The stable manifold of $z$ in a hyperbolic set $\Lambda$ of an endomorphism is defined likewise: 
\[W^s(z)= \{z'\in M:\; \lim_{ n\to+\infty} d(f^n(z),f^n(z'))= 0\}.\] 

The unstable manifold depends on the preimages.  For every orbit $\underline z= (z_n)_{n\in \Z} \in \Lambda^\Z$, has an unstable manifold:  
\[W^u(\underline z)= \{z'\in M:\; \exists (z'_n)_n \text{ orbit s.t. } \lim_{ n\to-\infty} d(z_n,z'_n)= 0\}.\]

If $f$ is a local diffeomorphism then $W^s$ and $W^u$ are immersed, but in general only $W^s$ is injectively immersed.

\begin{defi}A hyperbolic set $\Lambda$ is a \emph{basic piece} if it is locally maximal.
\end{defi}

\begin{exam}[Blender]\label{blender}
A blender of surface endomorphism is a basic set so that $C^1$-robustly its local unstable manifold cover an open subset of the surface. 

For instance let $I_-$ and $I_+$ be two disjoint segments of $[-1,1]$, and let $Q$ be a map which sends affinely each of these segments onto $[-1,1]$. 
 This is the case for instance of the following map:
\[(x,y)\in [-1,1]^2\mapsto \left\{\begin{array}{cl}
(Q(x), (2 y+1)/3)& x\in I_+\\
 (Q(x), (2 y-1)/3)& x\in I_-\end{array}\right. \]

\begin{figure}[h!]
	\centering
		\includegraphics[width=7cm]{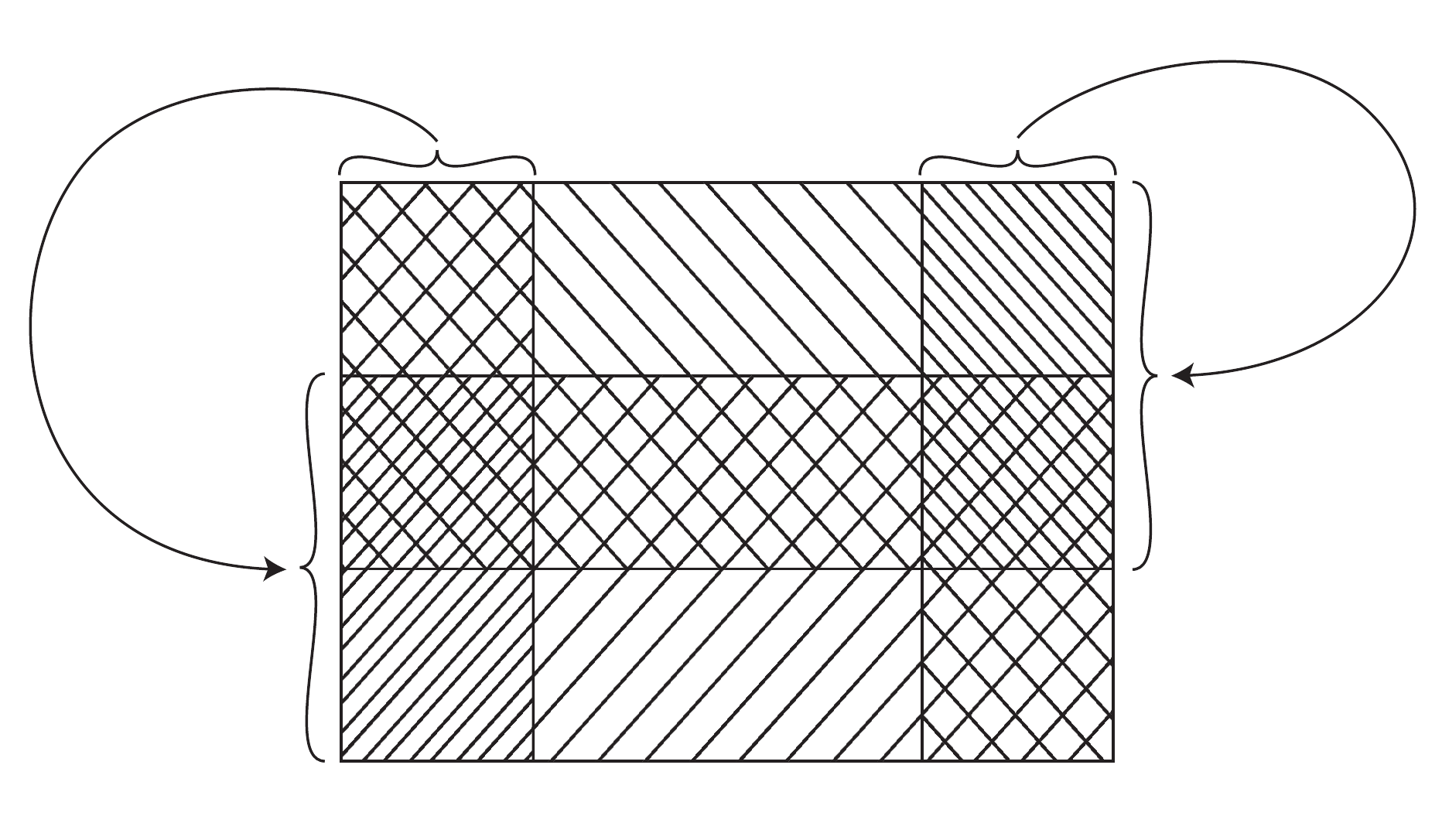}
	\caption{Blender of a surface local diffeomorphism}
\end{figure}

\end{exam}    
\begin{defi}An endomorphism satisfies  \emph{axiom A} if its non-wandering set is a finite  union of basic pieces.\end{defi}

\section{Properties of structurally stable dynamics}\label{partstabstruct}
\label{secStabimplieshyp}
Let us sate some definitions and conjectures on the structural stability in the $C^r$-category, $1\le r\le \infty$ and in the holomorphic category denoted by $\mathcal H$. 
Let $\mathcal C$ be a category in $\{C^r: 1\le r\le \infty\} \cup \{\mathcal H\}$. 

 \begin{defi}[Structural stability] A $\mathcal C$-map $f$ is \emph{structurally stable}  if every $\mathcal C$-perturbation $f'$ of the dynamics is conjugated: there exists a homeomorphism $h$ of the manifold so that $h\circ f= f'\circ h$. 
\end{defi}
A weaker notion of structural stability focuses on the non-wandering set $\Omega_f$ of the dynamics $f$.
\begin{defi}[$\Omega$-stability] A $\mathcal C$-map $f$ is \emph{$\Omega$-stable}  if for every $\mathcal C$-perturbation $f'$ of $f$, the dynamics of the restriction of $f$ to $\Omega_f$ is conjugated (via a homeomorphism) to the restriction of $f'$ to its non-wandering set $\Omega_{f'}$.
\end{defi}

We recall that an axiom A diffeomorphism $f$ satisfies the \emph{strong transversality condition} if its stable and unstable manifolds intersect transversally. Here is an outstanding conjecture:
\begin{conj}[Palis-Smale  structural stability conjecture, 1970 \cite{PaSm68}]
A $\mathcal C$-diffeomorphism is structurally stable if and only if it satisfies  axiom A and the strong transversality condition. 
\end{conj}
%

For complex rational maps of the sphere, this conjecture takes the form:
\begin{conj}[Fatou Conjecture, 1920]
Structurally stable quadratic map are those which  satisfy  axiom A and whose critical points are not periodic.
\end{conj}
Actually the initial Fatou conjecture stated the density of axiom A quadratic map. However, in section \ref{secStabimplieshyp},  we will recall the works of Ma\~ n\'e-Sad-Sullivan \cite{MSS} and Lyubich \cite{Ly84} showing the existence of an open and dense set of structurally stable rational maps. This implies the equivalence between the original Fatou Conjecture and the above conjecture.
 Among real quadratic maps, this conjecture\footnote{
The Fatou conjecture is implied by the Mandelbrot Locally connected (MLC) conjecture that we will not have the time to recall in this manuscript.}
 has been proved by Graczyk-Swiantek \cite{GS97} and Lyubich \cite{Ly97}. 
 

The description of $\Omega$-stable maps involves the no-cycle condition. We recall that any axiom A diffeomorphisms displays a non-wandering set $\Omega$ equal to a finite union of basic pieces $\Omega= \sqcup_i \Omega_i$.  The family  $(\Omega_i)_i$ is called the \emph{spectral decomposition}.  

 \begin{defi}[No-cycle condition]
An axiom A  diffeomorphism satisfies the \emph{no-cycle condition} if given $\Omega_1,\Omega_2, \dots, \Omega_n$ in the spectral decomposition, if $W^u(\Omega_i)$ intersects $W^s(\Omega_{i+1})$ for every $i<n$ and if 
$W^u(\Omega_n)$ intersects $W^s(\Omega_{1})$, then $\Omega_1=\Omega_2= \cdots = \Omega_n$.
\end{defi}

\begin{conj}[Smale $\Omega$-Stability Conjecture, \cite{Sm68}]
A $\mathcal C$-diffeomorphisms is structurally stable if and only if it satisfies  axiom A and the no-cycle condition.
\end{conj}

If the above conjectures turn out to be true then they would display a satisfactory description of structurally stable dynamics (for the axiom A diffeomorphisms are very well understood).

Let us define the probabilistic structural stability, which is implied by the 
 $\Omega$-stability. The definition involves the regular subset $\mathcal R_f$ of $\Omega_f$. This subset is formed by the points $p\in \Omega_f$ so that for every $a\in \{s,u\}$, there exist $\epsilon>0$ and a sequence of periodic points $(p_n)_n$ satisfying:
 \begin{itemize}
 \item $(p_n)_n$ converges to $p$,
 \item $(W^a_{\epsilon}(p_n))_n$ is relatively compact in the $\mathcal C$-topology.
\end{itemize} 
We showed in \cite{BD14} thanks to Katok's closing Lemma, that the set $\mathcal R_f$ has full measure for every ergodic, hyperbolic probability measure.
\begin{defi} A $\mathcal C$-map $f$ is probabilistically structurally stable  if for every $\mathcal C$-perturbation $f'$ of $f$,  the restriction of $f$ to $\mathcal R_f$ is conjugated to the restriction of $f'$ to its regular set $\mathcal R_{f'}$.
\end{defi} 
It is rather easy to see that probabilistic structural stability implies weak stability:
\begin{defi} A map $f$ is $\mathcal C$-weakly stable  if every $\mathcal C$-perturbation $f'$ of $f$ displays only hyperbolic  periodic points.
\end{defi}

To sum it up, the above definitions are related as follows:
\begin{center}
$\Omega$-Stability ${\Rightarrow}$
Probabilistic  Stability $\Rightarrow $ Weak Stability
\end{center}

\paragraph{The Lambda Lemma Conjecture.} This conjecture states that weak stability implies $\Omega$-stability. For the category of rational functions of the Riemannian sphere, this Lemma has been shown independently by Ma\~n\'e-Sad-Sullivan \cite{MSS} and Lyubich  \cite{Ly84}. 

As the space of rational functions is finite dimensional, a neighborhood of a rational function $f$ can be written as an analytic family $(f_\lambda)_{\lambda\in \D^n}$, with $\D$ the complex disk and $f_0=f$. If $(f_\lambda)_\lambda$ consists of weakly stable maps, then every periodic point $p_0$ of $f_0$ persists to as unique periodic point $p_\lambda$ for $f_\lambda$. Moreover the map $\lambda\mapsto p_\lambda$ is holomorphic. The Lambda lemma asks the following question. Given $p_0$ in closure $J^*_0$ of the set of periodic points of $f_0$, for every  sequence $(p^n_0)_n$ of periodic points converging to $p_0$, does the family $(\lambda \mapsto p^n_\lambda)_n$ converges? If yes, the\emph{ holomorphic motion is said well defined at $p_0$}.
\begin{lemm}[Lambda-Lemma, Ma\~n\'e-Sad-Sullivan \cite{MSS} and Lyubich  \cite{Ly84}]\label{lambda1}
If $(f_\lambda)_\lambda$ is weakly stable, then the holomorphic motion is well defined at every point $p_0\in J^*_0$.
\end{lemm}

We recall that every rational function $J^*$ is equal to the non-wandering set and that any attracting  periodic point displays a critical point in its basin. Furthermore if a rational function is not weakly stable, it displays a new attracting periodic point after a perturbation of the rational function. Hence the new critical point belongs to the basin of this attracting periodic point.  As the number of critical points is finite, after a finite number of perturbations the rational function turns out to be  weakly stable. This shows that weak stability is open and dense among the rational functions. By the Lambda Lemma \ref{lambda1}, this implies:
\begin{thm}[Ma\~n\'e-Sad-Sulivan  \cite{MSS}, Lyubich \cite{Ly84}]
There is an open and dense subset of rational functions of degree $d\ge2$ which are  $\Omega$-stable.
\end{thm}
This result enables them to deduce a stronger result: the density of the set of structurally stable rational functions.

We recall that a polynomial automorphism of $\C^2$ is a polynomial mapping of $\C^2$ which is invertible and whose inverse is polynomial. Among polynomial automorphisms of $\C^2$, Dujardin and Lyubich \cite{LD13} showed that the holomorphic motion is well defined on any uniformly hyperbolic compact set. We improved this result:
\begin{lemm}[Berger-Dujardin \cite{BD14}]\label{lambdalemC2}
If $(f_\lambda)_\lambda$ is a weakly stable family of polynomial automorphisms of $\C^2$,  the holomorphic motion is uniquely defined on the regular set $\mathcal R_0$ of $f_0$.
\end{lemm}
An immediate consequence of this result is that weak stability implies 
probabilistic stability for the category of polynomial automorphisms of $\C^2$. 

Unfortunately, there is no hope to get the density of $\Omega$-stable polynomial automorphisms of $\C^2$  because  in a non-empty open set \cite{Bu97} of the parameter space is formed by automorphisms displaying a wild horseshoe. However, we will see below that if none pertubations of the dynamics display a homoclinic  tangency, then the dynamics is weakly stable (under a mild hypothesis of dissipativeness). 
 
%


As a corollary of the techniques, we showed that one connected component of the set of weakly stable polynomial automorphisms is formed by  those which satisfy axiom A.   
%

\paragraph{The Ma\~n\'e  Conjecture}
In 1982, Ma\~ne conjectured in \cite{Ma82} that every $C^r$-weakly stable diffeomorphism satisfies axiom A for every $1\le r\le \infty$. He proved this conjecture for $r=1$ and surface diffeomorphisms.  Ma\~ n\'e developed this technology to prove that $C^1$-structurally stable diffeomorphisms satisfy  axiom A and the strong transversality condition in \cite{Ma88}. This work enabled also Palis to prove the same direction for the $C^1$-$\Omega$-stability conjecture \cite{Pa88}.
By developing Ma\~n\'e's works, Aoki and Hayashi proved the Ma\~n\'e conjecture for $r=1$ in any dimension \cite{Ao92,Ha92}.
 
\begin{center}
Weak Stability $\overset{\text{Man\~n\'e\; Conj.}}{\Longrightarrow}$ axiom A. 
\end{center}

After the next section, it will be clear for the reader that the Ma\~n\'e Conjecture implies the Lambda Lemma Conjecture in any category $\mathcal C$.

\paragraph{A Palis Conjecture}
We recall that a hyperbolic periodic point displays a \emph{homocline tangency} if its stable manifold $W^s(p)$ is tangent to its unstable manifold. Two saddle periodic points $p,q$ display a \emph{heterocline tangency} if $W^s(p)$ intersects transversally $W^u(q)$ whereas $W^s(q)$ is tangent to $W^u(p)$ (or vice versa). It is not hard to show that if a $C^r$-map is weakly stable then it cannot display a homoclinic nor a heteroclinic tangency, for every $1\le r\le \infty$. The same is true for one dimensional complex maps. For polynomial automorphisms of $\C^2$, it is a theorem  \cite{Bu97}.

Let us recall also a famous Conjecture of Palis \cite{Pa00} which states that if a dynamics which cannot be perturbed to one which displays a homoclinic nor a heteroclinic tangency, then it satisfies axiom A:
\begin{center}
Weak Stability $\Rightarrow $ Far from tangencies $\overset{\text{Palis\; Conj.}}{\Longrightarrow}$ axiom A. 
\end{center}
In the category of $C^1$-surface diffeomorphisms, this conjecture has been proved by Pujals-Sambarino \cite{PS00}. In the category of $C^1$-diffeomorphisms of higher dimensional manifolds, a weaker version has been proved by Crovisier-Pujals \cite{CP15}.
 
We notice that the Palis conjecture implies the Ma\~n\'e conjecture and so the Lambda lemma conjecture.
 
 \paragraph{A description of structurally stable dynamics as those far from tangencies?}
This question is widely open in the $C^r$-category for $r>1$ (for $C^1$-surface diffeomorphisms it is a consequence of Ma\~n\'e's theorem). It is also correct for the category of rational functions. 
This might be correct for polynomial automorphisms of $\C^2$.  Indeed, most of the work of Dujardin-Lyubich was dedicated to prove the following result:
\begin{thm}[Dujardin-Lyubich \cite{LD13}]
Given a polynomial automorphism $f$ of (dynamical) degree $d\ge 2$ and so that $ |det\, Df_0|\cdot d^2<1$, either $f$ is weakly stable, either a perturbation of $f'$ admits a homoclinic tangency. 
 \end{thm}
From Lambda Lemma \ref{lambdalemC2} we deduced:
\begin{cor}[Berger-Dujardin \cite{LD13}]
Given a polynomial automorphism $f$ of (dynamical) degree $d\ge 2$ and so that $ |det\, Df_0|\cdot d^2<1$, either $f$ is probabilistically stable, either a perturbation of $f'$ displays a homoclinic tangency. 
 \end{cor}
 
Let us stress that this direction might be interesting since numerically we can  see some local stable and unstable manifolds and observe if they display tangencies. 

%
%
%

\begin{figure}[h!]
	\centering
		\includegraphics[width=11cm]{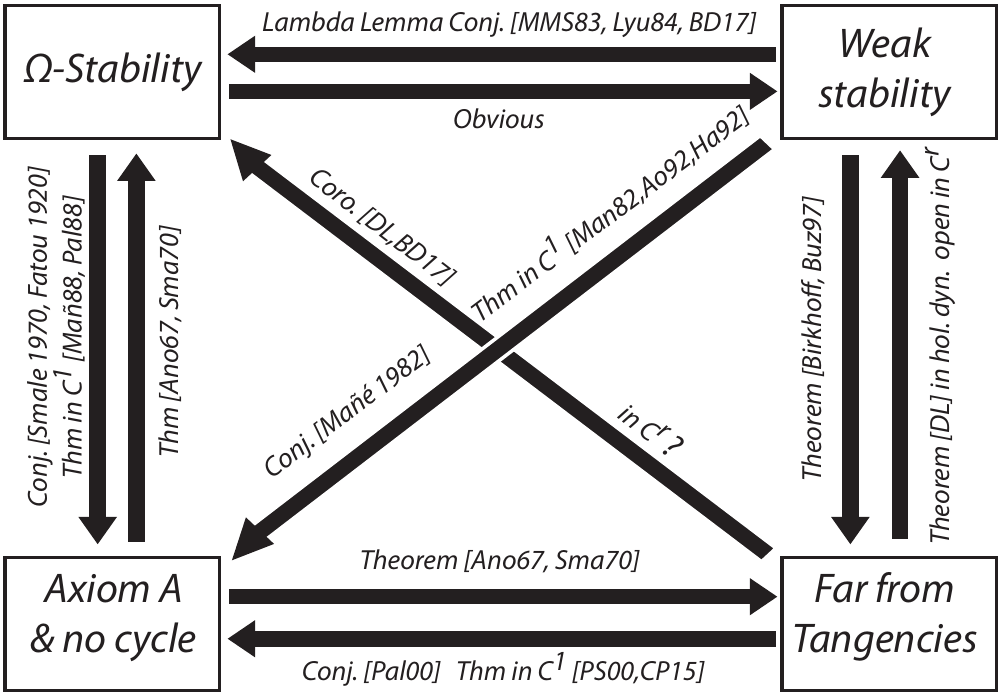}
	\caption{Summary of some Theorems and Conjectures on Structural Stability}
\end{figure}

\section{Hyperbolicity implies structural stability}\label{hypimpliesstab}

In the following subsection, we recall the proof ideas of several basic theorems showing the structural stability of subsets from hyperbolic hypotheses. 

\subsection{$\Omega$-stability of maps satisfying axiom A and the no-cycle condition}
\label{hypimpliesstab1}
First let us recall a generalization of the notion of structural stability for invariant subsets. 
\begin{defi}[Structurally stable subset] A compact set $\Lambda$ left invariant by a differentiable map $f$ of a manifold $M$ is \emph{structurally stable} if for every $C^r$-perturbation $f'$ of $f$, there exists a continuous injection $i\colon \Lambda \to M$ so that 
$f'\circ i = i\circ f$.
\end{defi}
We notice that $M$ is structurally stable if and only if $f$ is structurally stable. 
\begin{thm}[Anosov \cite{An67}, proof by Moser \cite{Mo69}]   \label{anosovthm}
A  uniformly hyperbolic compact set $\Lambda$ for a $C^1$-diffeomorphisms is structurally stable.
\end{thm}
\begin{proof} 
We want to solve the following equation:
\begin{equation}\tag{$\star$} f'\circ h\circ f^{-1} =  h \; .
\end{equation}
for $f'$ $C^1$-close to $f$ and $h$ $C^0$-close to the canonical inclusion $i\colon \Lambda\hookrightarrow M$. We shall use the implicit function theorem with the map:
$$\Phi \colon (h, f')\in C^0(\Lambda,M)\times C^1(M,M) \to 
f'\circ h\circ f^{-1}  \in C^0(\Lambda,M)\; .$$
We notice that $\Phi$ is a $C^1$-differentiable map of Banachic manifolds. Moreover it satisfies $\Phi(i,f)=i$.  
Hence to apply the implicit function theorem it suffices to prove that $id- \partial_h  \Phi(i,f)$ is an isomorphism.

Note that  the tangent space of the Banachic manifold $C^0(\Lambda,M)$ at the canonical inclusion $i$ is the following Banach space:
\[\Gamma := \{ \gamma \in C^0(\Lambda, TM): \forall x \in \Lambda\quad \gamma (x) \in T_{x} M\}.\]
The partial derivative of $\partial_h \Phi$ at $(i,f)$ is: 
$$\Psi:=\partial_h \Phi (i, f)\colon \sigma \in \Gamma\mapsto 
Df \circ  \sigma\circ f^{-1} \in \Gamma\; .$$
To compute the inverse of $id-\Psi$, we split $\Gamma $ into two $\Psi$-invariant subspaces $\Gamma = \Gamma^{u}\oplus\Gamma^{s}$, with:
\[\Gamma^{u}:= \{ \gamma \in C^0(\Lambda, TM): \forall x \in \Lambda\quad \gamma (x) \in E^{u}_x\}\quad \text{and}\quad \Gamma^{s}:= \{ \gamma \in C^0(\Lambda, TM): \forall x \in \Lambda\quad \gamma (x) \in E^{s}_x\}.\]

As the norm of $\Psi|\Gamma^{s}$ is less than $1$,  the map $(id-\Psi)|\Gamma^{s}$ is invertible with inverse equal to  $$\sum_{n\ge 0} (\Psi|\Gamma^{s})^n\; .$$ 

As $\Psi|\Gamma^{u}$ is invertible with contracting inverse, the map $(id-\Psi)|\Gamma^{u}$ is invertible with inverse: 
$$ -(\Psi|\Gamma^{u})\circ  (id-(\Psi|\Gamma^{u})^{-1})= - (\sum_{n\ge 1} (\Psi|\Gamma^{u})^{-n})\; .$$

Hence by the implicit function theorem, for every $f'$ $C^1$-close to $f$, there exists a continuous map $h$ $C^0$-close to $i$ which semi-conjugates the dynamics:
$$ f'\circ h = h \circ f \; .$$

As $i$ is injective and close to $h$, if $h(x)= h(y)$ then $x$ and $y$ are close. Also by semi-conjugacy, $h \circ f^n(x)= h\circ f^n(y)$ for every $n\in \Z$. Hence $f^n(x)$ is close to $f^n(y)$ for every $n$. 
By expansiveness (see below), we conclude that $x=y$ and so that $h$ is injective.
\end{proof}
\begin{lemm}[Expansiveness]\label{expansiveness}
Every hyperbolic compact set $\Lambda$ for a diffeomorphism is \emph{expansive}: there exists $\epsilon>0$ so that if two orbits $(x_n)_{n\in \Z}$ and $(y_n)_{n\in \Z}$ are uniformly $\epsilon$-close, then $x_0=y_0$.
\end{lemm}
\begin{proof}
First we notice that for $\epsilon$ small enough, given two such orbits,  $W^s_{2\epsilon} (y_n)$ intersects $W^u_{2\epsilon} (y_n)$  at a unique point $z_n$. We observe that $(z_n)_n$ is an orbit. As $f$ is expanding along $W^u_{2\epsilon} (y_n)$ for every $n$ and since $z_n\in W^u_{2\epsilon} (y_n)$, it comes that $z_n=y_n$ for every $n\ge 0$. Using the same argument for $f^{-1}$, it comes that $z_n=x_n$  for every $n\le 0$ and so $x_0=y_0$.
\end{proof}

The image $\Lambda(f'):= h(\Lambda)$ is called the \emph{hyperbolic continuation} of $\Lambda$. Since the density of periodic points is preserved by conjugacy, it comes: 
\begin{cor}
If $\Lambda$ is a basic piece, then its hyperbolic continuation is also a basic piece.
\end{cor}
A similar result has been proved by Shub during his thesis:
\begin{thm}[Shub \cite{Shub69}]
An expanding compact set $\Lambda$ for an endomorphisms $f$ is $C^1$-structurally stable.
\end{thm}
\begin{proof} First let us notice that $f$ is a local diffeomorphism at a neighborhood of the compact set $\Lambda$. Hence there exists $\epsilon>0$ so that for every $f'$ $C^1$-close to $f$, for every $x\in \Lambda$, the restriction $f'|B(x,\epsilon)$ is invertible. This enables us to look for a semi-conjugacy thanks to the map:
$$\Phi \colon (h, f')\in C^0(\Lambda,M)\times C^1(M,M) \to 
(f'|B(x,\epsilon))^{-1}\circ h\circ f  \in C^0(\Lambda,M)$$
The latter is well defined and of class $C^1$ on the $\epsilon$-neighborhood of the pair of the canonical inclusion $i\colon \Lambda\hookrightarrow M$ with $f$. Furthermore, it holds $\Phi(i,f)= i$ and the following partial derivative is contracting, with $\Gamma$ the tangent space of $C^0(\Lambda,M)$ at $i$.
 $$\partial_h \Phi (i, f) \colon  \sigma \in \Gamma \to 
Df^{-1}\circ \sigma \circ f  \in \Gamma\; .$$
Thus, by the implicit function Theorem, for $f'$ $C^1$-close to $f$, 
there exists a unique solution with $h\in C^0(\Lambda,M)$ close to $i$ for the semi-conjugacy equation:
\[\Phi(h,f')=h\Leftrightarrow h\circ f= f'\circ h\;.\]

As $h$ is close to the canonical inclusion, if $h(x)=h(x')$ then $x$ and $x'$ must be close.  Also by semi-conjugacy, it holds $h(f^n(x))=h(f^n(x'))$ for every $n\ge 0$. Thus the orbits $(f^n(x))_{n\ge 0}$ 
and $(f^n(x'))_{n\ge 0}$ are uniformly close. By forward expansiveness  (see below), it comes that $x=x'$.
\end{proof}
One easily shows by a similar argument to Lemma \ref{expansiveness}:
\begin{lemm}[Forward expansiveness]
Every expanding compact set $\Lambda$ is \emph{forward expansive}:
there exists $\epsilon>0$ so that if two orbits $(x_n)_{n\ge 0}$ and $(y_n)_{n\ge 0}$ are uniformly $\epsilon$-close, then $x_0=y_0$.
\end{lemm}

The two latter theorems enable us to explain the proofs of  Smale and Przytycki on {$\Omega$-stability}. We recall that the local stable and unstable manifolds of the points of a hyperbolic set $\Lambda$ for an endomorphism $f$ (which might display a non-empty critical set) are uniquely defined, provided that:
\begin{itemize}
\item Either $f|\Lambda$ is bijective,
\item Either  $\Lambda$ is injective.
\end{itemize}
On the other hand, the local stable manifold are always uniquely defined. Hence under these assumption, by looking at their images or preimages, the following is uniquely defined for $\epsilon>0$ small enough:
\[ W^s_\epsilon (\Lambda) = \cup_{x\in \Lambda} W^s_\epsilon (x)\quad 
W^u_\epsilon (\Lambda) = \cup_{x\in \Lambda} W^u_\epsilon (x)
\quad 
W^s (\Lambda) = \cup_{n\ge 0} f^{-n} (W^s_\epsilon (\Lambda))
\; .\]
The following generalizes Smale's definion of axiom A diffeomorphisms:
\begin{defi}[Axiom A in the sens of Przytycki]
A $C^1$-endomorphism satisfies \emph{axiom A-Prz}, if  its non-wandering set $\Omega$ is equal to the closure of the set of periodic points (or equivalently locally maximal), and if it is the disjoint union of an expanding compact set with a bijective, hyperbolic compact set. 
\end{defi}
For such maps we can generalize the notion of \emph{spectral decomposition}. Indeed 
by local maximality and compactness, the non-wandering set $\Omega$ of such maps is the finite union of (maximal) transitive subsets $\Omega_i$ called \emph{basic pieces}:
\[\Omega= \sqcup_i \Omega_i\; .\]
The family  $(\Omega_i)_i$ is called the \emph{spectral decomposition} of the axiom A-Prz endomorphism. Let us generalize the  no-cycle condition for such endomorphisms. 
 \begin{defi}[No-cycle condition]
An axiom A-Prz,  $C^1$-endomorphism satisfies the \emph{no-cycle condition} if given $\Omega_1,\Omega_2, \dots, \Omega_n$ in the spectral decomposition, if $W^u_\epsilon(\Omega_i)$ intersects $W^s(\Omega_{i+1})$ for every $i<n$ and if 
$W^u_\epsilon(\Omega_n)$ intersects $W^s(\Omega_{1})$, then $\Omega_1=\Omega_2= \cdots = \Omega_n$.
\end{defi}
F. Przytycki  generalized Smale's Theorem on the $\Omega$-stability of axiom A diffeomorphisms which satisfy the no-cycle condition as follows:
\begin{thm}[\cite{Sm68}, \cite{Pr77}]
If a $C^1$-endomorphism satifies axiom A-Prz and the no-cycle condition, then it is $C^1-\Omega$-stable.
\end{thm}
\begin{proof}[Sketch of proof of the Smale's $\Omega$-stability Theorem]
First let us recall that by Anosov Theorem, the non-wandering set $\Omega$ is structurally stable, and its hyperbolic continuation is still locally maximal (for a neighborhood uniformly large among an open set of perturbations of the dynamics).

Then the no-cycle condition is useful to construct a filtration $(M_i)_i$:
\begin{prop}
If an axiom A, $C^1$-diffeomorphism $f$ satisfies the no-cycle condition, then there exists a chain of open subsets:
\[\emptyset = M_0 \subset M_1\subset \cdots \subset M_N=M\]
 so that 
  $f(M_i)\Subset M_i$ and $\Omega_i \Subset M_{i}\setminus M_{i-1}$ for every $i\ge 1$.
 \end{prop}
The proof of this proposition involve Conway Theory and can be find 
in \cite[Thm 2.3 p. 9]{shubstab78}.

By using this filtration and the (uniform) local maximality of the hyperbolic continuation of the non-wandering set, one easily deduces the $\Omega$-stability. 
%
%
\end{proof}


\subsection{Structural stability of dynamics satisfying axiom A and  the strong transversality condition}\label{hypimpliesstab2} 
\paragraph{Structural stability of diffeomorphisms}
We recall that an axiom A diffeomorphism satisfies the {strong transversality condition} if for any  non-wandering points $x$ and $y$, the stable manifold of $x$ is transverse to the unstable manifold of $y$.

\begin{rema} By using the inclination lemma, one easily shows that the strong transversality condition implies the no-cycle condition.
\end{rema}

The following theorem generalizes Anosov Theorems \ref{anosovthm}:

\begin{thm}[Robbin \cite{Ro71}, Robinson \cite{Ro76}] \label{SSTRob}
For every $r\ge 1$, the diffeomorphisms which satisfy axiom A and the strong transversality condition are $C^r$-structurally stable.\end{thm}
Let us recall that the Ma\~ne theorem \cite{Ma88} implies that a $C^1$-structurally stable diffeomorphism satisfies also  axiom A and the strong transversality condition, and so both solve the conjecture   of $C^1$-structural stability. 

We will state Conjecture \ref{ConjPrz} generalizing this theorem for local diffeomorphisms. Hopefully the following will help the reader to tackle it. 
\begin{proof}[Sketch of proof of Theorem \ref{SSTRob}]
Again we want to solve the following semi-conjugacy equation:
\begin{equation}\tag{$\star$} f'\circ h\circ f^{-1} =  h 
\end{equation}
for $f'$ $C^1$-close to $f$ and $h$ $C^0$-close to the identity of $M$. 

For  $f'=f$ and $h=id$, Equality ($\star$) is valid. The set of perturbations of the identity is isomorphic to  $\Gamma = \{ \gamma \in C^0(M, TM): \forall x \in M\quad \gamma (x) \in T_{x} M\}$ by using the exponential map (associated to a Riemannian metric on $M$). Let 
$\tilde f := u\in T_xM \mapsto \exp_{f(x)}^{-1}\circ  f \circ \exp_x(u)$.

Then Equation $(\star)$ is equivalent to:
\begin{equation}\tag{$\star\star$} \tilde f'\circ \sigma\circ f^{-1} = \sigma, \quad \text{for }\sigma \in \Gamma \quad C^0\text{-small.}
\end{equation}

As the map 
$\Phi \colon (\sigma, f')\in \Gamma \times C^1(M,M) \to 
\Phi_{f'}(\sigma)=\sigma -\tilde f'\circ \sigma \circ f^{-1}  \in \Gamma $ is of class $C^1$, 
and vanishes at $(0,f)$,  
we shall show that $\partial_h  \Phi$ is left-invertible. 

Let  
$$\Psi:=\partial_h  \Phi(0,f)\colon \sigma \in \Gamma\mapsto \sigma-
Df \circ  \sigma\circ f^{-1}   \in \Gamma \; .$$


The following is shown in \cite{Ro71}:
\begin{prop}\label{sectionEi}
For every $i$, there exists a neighborhood $N_i$ of $\Omega_i$ and  continuous extension $E^s_i$ and $E^u_i$ of respectively $E^s|\Omega_i$ and $E^u|\Omega_i$ to $N_i$, so that:
\begin{itemize}
\item There exists a filtration $(M_i)_i$ adapted to $(\Omega_i)_i$ so that $N_i= M_i\setminus M_{i-1}$. The subsets  $(N_i)_i$ form an open covering of $M$,
\item if $x\in N_i\cap f^{-1}(N_j)$, with $j\le i$, then $Df(E^s_i(x))\subset E^s_j(f(x))$, and $Df(E^u_i(x))\supset E^u_j(f(x))$.
\end{itemize}
\end{prop}

Let $(\gamma_i)_i$ be a partition of the unity adapted to $(N_i)_i$.

For every $i$ let $p^s_i$ and $p^u_i$ be the projections onto respectively $E^s_i$ and $E^u_i$ parallely to  $E^u_i$ and $E^s_i$. 

Given $x\in M$ and $v\in T_x M$, we put 
$v_i^s:= \gamma_i \cdot p^s_i(v)$ and   $v_i^u:= \gamma_i \cdot p^u_i(v)$. 
We observe that $v= \sum_i v_i^s+v_i^u$. Thus
$Df(v)= \sum_i Df(v_i^s)+Df(v_i^u)$. 
As $(Df^n(v_i^s))_{n\ge 0}$ and $(Df^{-n}(v_i^u))_{n\ge 1}$ converge exponentially fast to $0$, we consider:

\[J\colon \sigma\in \Gamma \mapsto \sum_i \sum_{n\ge 0} Df^n(\sigma_i^s\circ f^{-n}(x))
-\sum_{n\ge 1} Df^{-n}(\sigma_i^u\circ f^{n}(x))\; .\]

We notice that $J$ is a left inverse of $\Psi$ :
$$ J\circ \Psi = id$$.

The following equations are equivalent:
$$ \Phi_{f'}(\sigma)=0\Leftrightarrow (\Phi_{f'}-\Psi)(\sigma)+\Psi(\sigma)=0,$$
$$ \Leftrightarrow J\circ (\Phi_{f'}-\Psi)(\sigma)+J\circ \Psi(\sigma)=J(0).$$
Now observe that $J(0)=0$ and $J\circ \Psi(\sigma)=\sigma$. Hence  $(\star\star)$ is equivalent to 
$$J\circ (\Psi-\Phi_{f'})(\sigma)=\sigma.$$

It is easy to see that whenever $f'$ is $C^1$-close to $f$, the map $\Phi_{f'}$ is $C^1$-close to $\Psi$ at a neighborhood of the $0$-section. Hence the map $J\circ (\Psi-\Phi_{f'})$ is contracting and sends a closed ball about the zero section into itself. The contracting mapping theorem implies the existence of a fixed point $\sigma$.
Hence $(\star)$ displays a solution $h=\exp\circ \sigma$ close to the identity in the space of continuous maps.

It remains to show that the semi-conjugacy $h$ is bijective. Contrarily to Anosov maps, in general  axiom A diffeomorphisms are not expansive and the semi-conjugacy is not uniquely defined. Hence Robbin brought  a new technique to construct a map $h$ which is bijective. He defined the following metric:
$$ d_f(x,y)= \sup_{n\in \mathbb Z} d(f^n(x),f^n(y))\;,$$
where $d$ is the Riemannian metric of the manifold $M$.


Let us just notice that if the semi-conjugacy  $h=\exp\circ \sigma$ satisfies that $\sigma$ is $C^0$-small and $d_f$-Lipschitz with a small constant $\eta$, then $h$ is injective.

Indeed if $h(x)=h(y)$, then by $(\star)$, $h(f^n(x))= h(f^n(y))$ for every $n$. Since $h$ is close to the identity, the orbits $(f^n(x))_n$ and $(f^n(y))_n$ are uniformely close, and so that $d_f(x,y)$ is small. As $\sigma$ is $\eta$-Lipschitz, it comes:
\[0= d(h(x),h(y))\ge d(x,y)-\eta d_f(x,y)\]    
The same holds at any $n^{th}$-iterate:
\[0=d(h(f^n(x)),h(f^n(y)))\ge d(f^n(x),f^n(y))-\eta d_f(f^n(x),f^n(y))=d(f^n(x),f^n(y))-\eta d_f(x,y)\;.\]
Let $n$ be such that $d(f^n(x),f^n(y))\ge  d_f(x,y)/2$. Then 
\[0=d(h(f^n(x)),h(f^n(y)))\ge (1-2\eta) d(f^n(x),f^n(y))\;.\]
Thus $f^n(x)=f^n(y)$ and so $x=y$.

To obtain the section $\sigma$ $d_f$-Lipschitz, Robbin assumed the diffeomorphism $f$ of class $C^2$. Then in Proposition \ref{sectionEi}, he constructs the section $(E^s_i)_i$ and  $(E^u_i)_i$ $d_f$-Lipschitz, so that the map $J$ preserves the  $d_f$-Lipschitz sections. On the other hand the map $\Psi-\Psi_{f'}$ diminishes the $d_f$-Lipschitz constant for $f'$ $C^1$-close to $f$. Therefore the map $J\circ (\Psi-\Psi_{f'})$ preserves the space of continuous sections with small $d_f$-Lipschitz constant, and so its fixed point enjoys a small $d_f$-Lipschitz constant.  

The $C^1$-case was handled by Robinson. His trick was to smooth the map $Df$ to a $C^1$-map $\tilde Df$, and to replace $Df$ by $\tilde Df$ in the definition of $\Psi$ to define $\tilde \Psi$. Then he defined  likewise $\tilde Df$-pseudo invariant sections $(\tilde E^s_i)_i$ which are $d_f$-Lipschitz. By replacing $(E^s_i)_i$ by $(\tilde E^s_i)_i$ in the definition of $J$, he defined a left inverse $\tilde J$ of $\tilde \Psi$. Then he showed likewise that the map $\tilde J\circ (\tilde \Psi-\Psi_{f'})$ admits a $C^0$-small, $d_f$-Lipschitz fixed point, which is a solution of $(\star\star)$. 
\end{proof}

\paragraph{Structural stability of covering.}
We recall that every local diffeomorphism of a compact (connected) manifold is a covering. 

F. Przytycki \cite{Pr77} introduced an example of surface covering suggesting the following \emph{strong transversality condition}.
 \begin{defi}
A covering map $f$ satisfies {\emph axiom A and the strong transversality condition} if:
\begin{enumerate}[(i)]
\item The non-wandering set is locally maximal.
\item  The non-wandering set $\Omega$ is the union of a hyperbolic set on which $f$ acts bijectively with a repulsive set.
\item $\forall x\in \Omega, \; \underline y^1,\dots,\underline  y^k)\in \overleftarrow \Omega$, the following multi-transversality condition holds:
$$W^s(x)\pitchfork W^u(\underline  y^1)\pitchfork \cdots \pitchfork W^u(\underline  y^k)\; .$$
\end{enumerate}
\end{defi}
We recall that a finite family of submanifolds $(N_i)_i$ is multi-transverse if $N_1$ and $N_2$ are transverse, $N_3$ is transverse to $N_1\cap N_2$, ..., and for every $i \ge 3$, $N_{i}$ is transverse to $N_1\cap N_2\cap \cdots \cap N_{i-1}$. We notice that $(iii)$ implies $(ii)$.

Here is a generalization of a conjecture of  Przytycki \cite{Pr77}:

\begin{conj}\label{ConjPrz}
The $C^1$-struturally stable coverings are those which satisfy  axiom A and the strong transversality condition.
\end{conj} 
The fact that structurally stable coverings are axiom A has been proved by Aoki-Moriyasu-Sumi \cite{AMS01}, and the strong transversality condition has been proved by Iglesias-Portela-Rovella. The other direction is still open in the general case.

This conjecture has been proved in two special cases. The first one solves the initial Przytycki conjecture for surface coverings:
\begin{thm}[Iglesias-Portela-Rovella \cite{IPR12}] If a covering map of a surface satisfies  axiom A and the strong transversality condition then it is $C^1$-structurally stable. 
\end{thm}
The other case is for attractor-repellor covering. 

\begin{thm}[Iglesias-Portela-Rovella \cite{IPR10}] Let $M$
be a compact manifold. If
$f$ is a $C^1$- covering map satisfying  axiom A, and so that its basic pieces are either bijective attractors or expanding sets, then $f$ is $C^1$-structurally stable. \end{thm}
The strong transversality condition for these maps is certainly satisfied since, the unstable manifolds are either included in the attractor or form open subset of the manifold. They gave the following example:
\[f\colon (z,z')\in \mathbb S^1\times \hat \C\mapsto( z^2, z/2+ z'/3),\]
where the non-wandering set consists of an expanding circle and of the Smale solenoid. 

In \cite{BK13}, we constructed $d_f$-Lipschitz plane fields for endomorphisms which satisfies  axiom A and the strong transversality condition. This might be useful to prove that under the hypothesis of Conjecture \ref{ConjPrz}, the following map has a left inverse:
\[\sigma\in \Gamma^0(TM) \mapsto \sigma - Df^{-1}\circ \sigma\circ f \in \Gamma^0(TM)\; .\]

\subsection{Structural stability of the inverse limit}\label{hypimpliesstab3}

Given an endomorphism $f$ of a compact manifold $M$, the inverse limit $\overleftarrow M_f$ of $f$ is the space of orbits :
\[\overleftarrow M_f:= \big\{\underline x= (x_n)_{n\in \mathbb Z} \colon x_{n+1} = f(x_n)\big\}\;.\]
It is a closed subset of $M^\mathbb Z$, which is compact endowed with the product metric:
\[d(\underline x,\underline x')= \sum_{n\in \mathbb Z} 2^{-|n|} d(x_n,x'_n)\;.\]
We notice that the inverse limit is homemorphic to $M$ when $f$ is a homeomorphism of $M$. 

We notice also that the shift dynamics $\overleftarrow f$ acts canonically on $\overleftarrow M_f$:
$$\overleftarrow f\colon (x_n)_n \mapsto (x_{n+1})_{n}. $$
With $\pi_0\colon (x_n)_n \mapsto x_0$ the zero coordinate projection, it holds:
$$\pi_0 \circ \overleftarrow f = f\circ \pi_0.$$

From this one easily deduces that the non-wandering sets  $\arr \Omega_{f}$ and $\Omega_{f}$ of respectively $\arr f$ and $f$ satisfies the following relation:
\[ \arr \Omega_{f} = \Omega_{f}^\Z\cap \arr M_f\; .\]

\begin{defi}
The endomorphism $f$ is $C^r$-inverse limit stable if for every $C^r$-perturbation $f'$ of $f$, there exists a homeomorphism $h$ from $\overleftarrow M_f$ onto $\overleftarrow M_{f'}$ so that:
\[h\circ \overleftarrow f = \overleftarrow f'\circ h.\]
\end{defi}

We can define the unstable manifold of every point $\underline x=(x_i)_i\in \overleftarrow \Omega_f$:
\[W^u(\underline x; \overleftarrow f):=\{\underline y=(y_i)_i\in \overleftarrow M_f\colon d(x_i,y_i)\to 0,\; i\to -\infty\}\]
When $f$ satisfies axiom A, it is an actual manifold embedded in $\overleftarrow M_f$. Moreover, the $0$-coordinate projection $\pi_0$ displays a differentiable restriction $\pi_0| W^u(\underline x; \overleftarrow f)$.

On the other hand, there exists $\epsilon>0$ so that the following local stable manifold is an embedded submanifold of $M$, for every  $x\in \Omega_f$:
  \[W^s_\epsilon( x;  f):=\{y\in M\colon d(f^n(x),f^n(y))\to 0,\; n\to +\infty\}\; .\]

In \cite{BR13}, we notice that surprisingly, for certain axiom A endomorphisms, the presence of critical set (made by points with non surjective differential) does not interfere with the $C^1$-inverse structural stability.  This leads us to define:
\begin{defi} An axiom A endomorphism $f$ satisfies the weak transversality condition if for every $\underline x\in \overleftarrow \Omega_f$ and every $y\in \Omega_f$, the map 
$\pi_0|W^u(\underline x; \overleftarrow f)$ is transverse to  $W^s_\epsilon(y)$.\end{defi}

There are many examples of endomorphisms which satisfy axiom A and the weak transversality condition. For instance:
\begin{itemize}
\item any axiom A map of the one point compactification $\hat \R$ of $\R$, in particular those of the form $x\mapsto x^2+c$ and even the constant map $x\mapsto 0$.
\item if $f_1$ and $f_2$ satisfy axiom A and the weak transversality condition, then the product dynamics $(f_1,f_2)$ do so.
\item By the two latter points, note that the map $(x,y,z)\mapsto (x^2,y^2,0)$ of $\R^3$ satisfies axiom A and the weak transversality condition.
\end{itemize}
The latter map is not at all structurally stable, for its critical set is not and intersects moreover the non-wandering set. For this reason the following conjecture might sound irrealistic:
\begin{conj}[Berger-Rovella \cite{BR13}]
The $C^1$-inverse limit stable endomorphisms are  those which satisfy axiom A and the weak transversality condition.
\end{conj}
However  in \cite{BR13}, we gave many evidences of veracity of this conjecture. Then in \cite{BK13} we showed one direction of this conjecture ; the other direction is still open.
\begin{thm}[Berger-Kocksard \cite{BK13}]\label{BK}
If a $C^1$-endomorphisms of a compact manifold satisfies  axiom A and the weak transversality condition, then it is inverse limit stable. 
\end{thm}
The proof of this theorem follows the strategy of the Robbin structural stability theorem. The main difficulty is the construction of pseudo-invariant plan fields $(E_i^s)_i$  and $(E_i^u)_i$, for the endomorphisms display in general a non-empty critical set.  

\begin{figure}[h!]
	\centering
		\includegraphics[width=9cm]{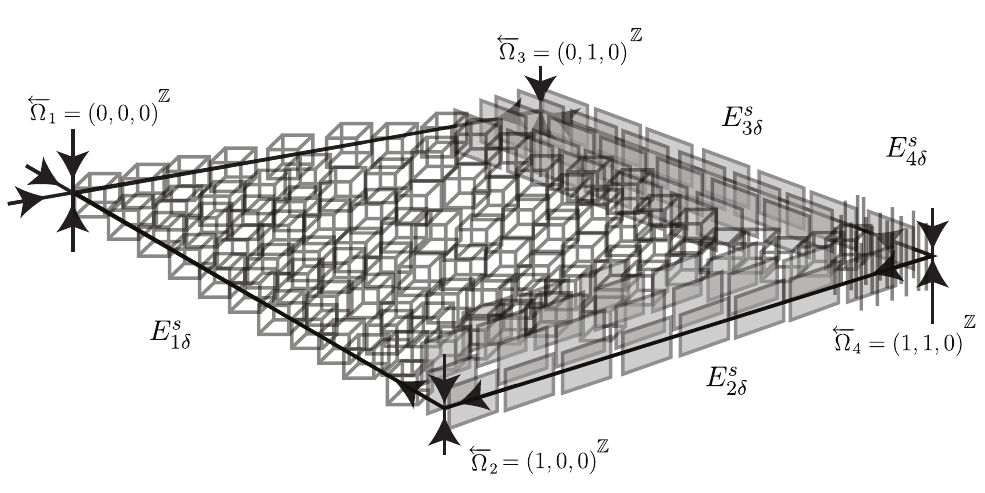}
	\caption{Construction of $(E_i^s)_i$ for the map $(x,y,z)\mapsto (x^2,y^2,0)$ }
\end{figure}
\begin{figure}[h!]
	\centering
		\includegraphics[width=9cm]{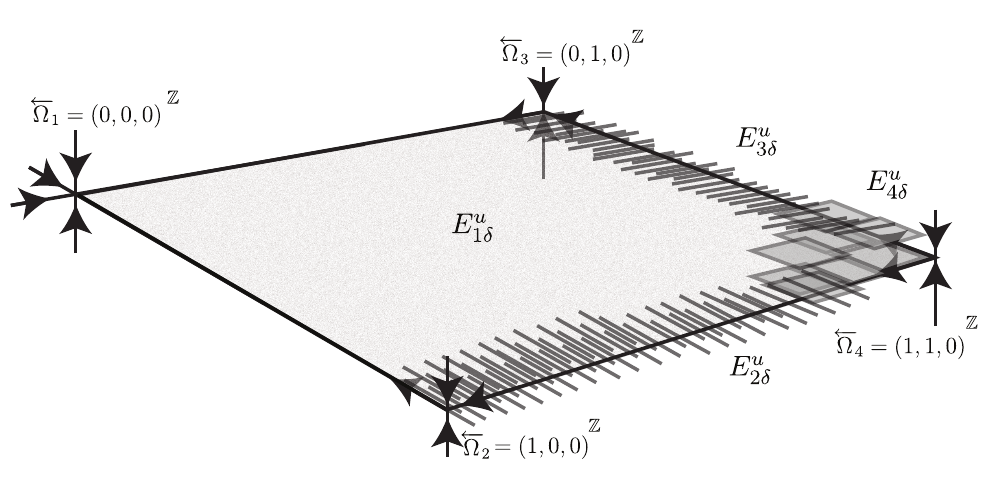}
	\caption{Construction of $(E_i^u)_i$ for the map $(x,y,z)\mapsto (x^2,y^2,0)$}
\end{figure}

\section{Links between structural stability in dynamical systems and  singularity theory}\label{sectionLinks}
In the last section we saw how the inverse stability does not seem to  involve any singularity theory. 
However let us notice that if a $C^\infty$-endomorphism of a manifold $M$ is structurally stable (that  is conjugated to its perturbation via a homeomorphism of $M$), then its singularities are \emph{$C^0$-equivalently, structurally stable}:
 \begin{defi}
 Let $f$ be a $C^\infty$-map from a manifold $M$ into a possibly different manifold $N$ and $r\in \{0,\infty\}$. The map $f$ is $C^r$-\emph{equivalently, structurally stable} if for every $f'$ $C^\infty$-close to $f$, there are  $h\in Diff^r(M)$ and $h'\in Diff^r(N)$ which are $C^r$-close to the identity and such that the following diagram commutes:
\[\begin{array}{lcccr} 
&&f'&&\\
 &M&\rightarrow &N&\\
 h&\uparrow &&\uparrow&h'\\
 &M&\rightarrow&N&\\
 &&f&&\end{array}.\]
\end{defi}

The equivalently, structural stability has been deeply studied, in particular by Whitney, Thom and Mather. We shall recall some of the main results, by emphasizing their similarities with those of structural stability in dynamical systems. 

 \subsection{Infinitesimal stability}
Let $M, N$ be compact manifolds.  For $r\in \{0,\infty\}$, let $\chi^r(M)$ and $\chi^r(N)$ be the space of $C^r$-sections of respectively $TM$ and $TN$.

 \begin{defi}
A Diffeomorphism  $f\in Diff^1(M)$ is $C^0$-\emph{infinitesimally stable} if the following map is surjective:
\[\sigma\in \chi^0(M)\mapsto Tf\circ \sigma- \sigma\circ f\in \chi^0(f),\]
with $\chi^0(f)$ the space of continuous sections of the pull back bundle  $f^*TM$. 
\end{defi} 
  In the Robbin-Robinson proofs of structural stability (Theorem \ref{SSTRob}), we saw the importance of the left-invertibility of $\sigma\mapsto Tf\circ \sigma- \sigma\circ f$. The latter implies the  $C^0$-{infinitesimal stability} which is equivalent to the $C^1$-structural stability:
\begin{thm}[Robin-Robinson-Ma\~ne \cite{Ro71},\cite{Ro76}, \cite{Ma88}]
The $C^0$-{infinitesimally stable} diffeomorphisms are the $C^1$-equivalently stable maps.
\end{thm}

A similar definition exists in Singularity Theory:
\begin{defi}
Let $f\in C^\infty(M,N)$ is $C^\infty$-\emph{ equivalently infinitesimally stable} if 
 the following map is surjective:
\[(\sigma, \xi)\in \chi^\infty(M)\times \chi^\infty(N)\mapsto Tf\circ \sigma- \xi\circ f\in \chi^\infty(f)\]
with $\chi^\infty(f)$ the space of $C^\infty$-sections of the pull back bundle  $f^*TM$. 
\end{defi}
It turns out to be equivalent to the $C^\infty$-equivalent stability.
\begin{thm}[Mather \cite{Ma170,Ma270,Ma370,Ma470,Ma570}]
The $C^\infty$-{ infinitesimally equivalently stable} maps are the $C^\infty$-equivalently stable maps.
\end{thm}
The latter might sound complicated to verify, but on concrete examples it is rather easy to check. That is why following Mather, it is a satisfactory description of $C^\infty$-equivalently structurally stable maps. 

 \subsection{Density structurally stable maps}

Let us point out two similar results on structural stability:
\begin{thm} [Thom, Mather \cite{Ma73, Ma76, GWPL}]\label{C0ESS}
For every manifolds $M,N$, the $C^0$-equivalently structural stable maps form an open and dense set in $C^\infty(M,N)$.
\end{thm}
Let us recall:
\begin{thm}[Ma\~ne-Sad-Sullivan \cite{MSS}, Lyubich \cite{Ly84}]
For every $d\ge 2$, the set of structurally stable rational functions is open and dense. 
\end{thm}
In both cases, we do not know  how to describe these structurally stable maps. 

Still the axiom A condition is a candidate to describe the structurally stable rational functions, since the famous Fatou conjecture (1920).  
On the other hand, there is not even a conjecture for the description of the  $C^0$-equivalently structural stable maps.
 
Following Mather, a nice way to describe the equivalently structural stable maps would be (a similar way to) the $C^\infty$-equivalently infinitesimal stability.

 Nevertheless,  Mather proved that $C^\infty$-equivalently infinitesimal stable maps  are  dense if and only if the dimensions of $M$ and $N$ are not ``nice" \cite{Ma670}. We define the nice dimensions below. Thus one has to imagine a new criteria (at least of in ``not nice" dimensions) to describe the $C^0$-equivalently structural stable maps. 

\begin{defi}[Nice dimensions]
If $m=dim \, M$ and $n= dim\, N$, the pair of dimensions $(m;n)$ is nice if and only if one of the following conditions holds:
\[\begin{array}{rcl}
n-m \ge  4 & \text{and}&  m<\frac 67 n+\frac 8 7,\\
3\ge n- m \ge  0 & \text{and}& m< \frac 6 7 n +\frac 9 7,\\
 n- m =-1& \text{and}& n <8,\\
 n- m =-2& \text{and}& n <6,\\
 n- m =-3& \text{and}& n <7.\end{array}\]
\end{defi}
We notice that if  $n:=dim\,M=dim\,N$, then the pair of dimensions $(m;n)$ is nice if and only if $n\le 8$.

Let us finally recall an open question: 
\begin{prob}
In nice dimensions,  does a 
$C^0$-equivalently structurally stable map is always $C^\infty$-equivalently structurally stable map? 
\end{prob}
\subsection{Geometries of the structural stability}
The proof of the Thom-Mather Theorem \ref{C0ESS} on the density of $C^0$-equivalently structurally stable involves the concept of stratification (by analytic or smooth submanifolds). 

Similarly, the set of stable and unstable manifolds of an axiom A diffeomorphisms form a stratification of laminations, as defined in \cite{BeMem}.  

Let us recall these definitions.

\subsubsection{Stratifications}

A \emph{ stratification} is the pair of a locally compact subset $A$ and a locally finite partition $\Sigma$ by locally compact subsets $X\subset A$, called {\emph strata}, and satisfying: 
  \[\forall (X,Y)\in \Sigma^2,\; cl(X)\cap Y\not=\emptyset\Rightarrow cl(X)\supset Y\; .\]
\[\mathrm{We\; write \; then \;} X\ge Y\; .\]

In practical, the set $A$ will be embedded into a manifold $M$, and the strata $X$ will be endowed with a structure of analytic manifold, differentiable manifold or even lamination, depending on the context. 

\subsection{Whitney Stratification}
The first use of stratification goes back to the work of Whitney to describe the algebraic  varieties. Then it has been generalized by Thom 
and Lojasiewicz for the study of analytic variety and even semi-analytic variety. 

\begin{defi}
\emph{An analytic variety} of $\R^n$  is  the zero set of an analytic function on an open subset of $\R^n$. \emph{An analytic submanifold} is a submanifold which is also an analytic variety. \emph{A semi-analytic variety} is a subset $A$ of  $\R^n$ which is covered by open subset $U$ satisfying:
\[A\cap U =  \cap_{i=1}^N\cup_{i=1}^N F_{ij}\]
with $F_{ij}$ of the form $\{q_{ij} >0\}$ or $\{q_{ij} =0\}$ and $q_{ij}$ a real analytic function on $U$. 
\end{defi}

\begin{thm}[Whitney-Lojasiewicz \cite{Lo70}]
Any semi-analytic variety $S\subset R^n$ splits into a stratification $\Sigma$ by analytic manifolds.
\end{thm}
  One important property of the semi-analytic category is its stability by projection from the Seidenberg Theorem: given any projection $p$ of $\R^n\to \R^p$, the image by $p$ of any  semi-analytic variety is a semi-analytic variety.

  \begin{figure}
	\centering
		\includegraphics[width=6cm]{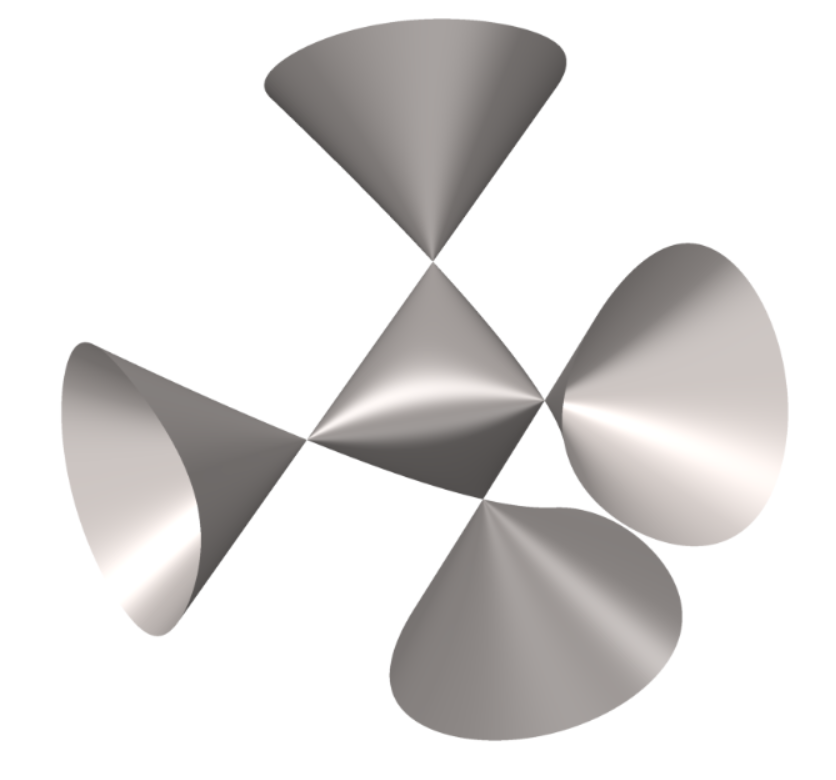}
	\caption{Algebraic variety $x^2+y^2+z^2+2xyz-1=0$}
\end{figure}

\subsection{Thom-Mather Stratification}
The following is a key step in the proof of the Thom-Mather Theorem \ref{C0ESS}. 
\begin{thm}[Thom-Mather]
For every $C^\infty$-generic map from a compact manifold $M$ into $N$, there exists a stratification on $\Sigma_M$ of $M$ and a stratification $\Sigma_N$ on $N$ such that:\begin{enumerate}[$(i)$]
\item The strata of $\Sigma_M$ and $\Sigma_N$ are smooth submanifolds,  
\item the restriction of $f$ to each stratum of $\Sigma_M$ is a submersion onto a stratum of $N$,
\item this stratification is structurally stable: for every perturbation $f'$ of $f$ there are stratifications $\Sigma_M'$ and $\Sigma_N'$ homeomorphic  to respectively  $\Sigma_M$ and $\Sigma_N$, so that $(ii)$ holds for $f'$.\end{enumerate}
\end{thm}

The proof of this theorem is extremely interesting, it involves in particular the jet space, Thom's transversality theorem and Whitney stratifications in semi-analytics geometry. 

\subsection{Laminar stratification}
Analogously to singularity theory, a structurally stable $C^1$-diffeomorphism displays a stratification. 
\begin{defi}
 A \emph  {lamination} of $M$ is a locally compact subset $\mathcal L$ of $M$, which is locally homeomorphic to the product of a $\mathbb R^d$ with a locally compact set $T$, so that $(\mathbb R^d \times \{t\})_{t\in T}$ corresponds to a continuous family of submanifolds.
\end{defi}

\begin{defi} A {\emph stratification of laminations} is a stratification  whose strata are endowed with a structure of  lamination.\end{defi}

\begin{prop}[\cite{BeMem}]
Let $f$ be a diffeomorphism  $M$ which satisfies Axiom A and the strong transversality condition. Then the stable set of every basic piece $\Lambda_i$ of $f$ has a structure of lamination $X_i$ whose leaves are stable manifolds. Moreover the family $\Sigma_s:= (X_i)_i$ forms a stratification of laminations such that  $X_i\le X_j$ iff $\Lambda_i \succeq \Lambda_j$ i.e. $W^u(\Lambda_i)\cap W^s(\Lambda_j)\not = \emptyset$.\end{prop}

\section{Structural stability of endomorphisms with singularities}
\label{Sec_SS_endo_w_Singu}
We are now ready to study the structural stability of endomorphisms which display a non empty critical set. 


In dimension $2$, Iglesias-Portela-Rovella \cite{IPR08} showed the structural stability of $C^3$-perturbations of the hyperbolic rational functions $f$ which are equivalently stable and whose critical sets do not self-intersect along their orbits, nor intersect the non-wandering set. 

In all these examples, the critical set does not self intersect along its orbit. J. Mather suggested me to generalize a study he did about structural stability of graph of maps.

Let $G:=(V,A)$ be a finite oriented graph with a manifold $M_i$ associated to each vertex $i\in V$, and with a smooth map $f_{ij}\in C^\infty(M_i,M_j)$ associated to each arrow $[i,j]\in A$ from $i$ to $j$.

For $k\in \{0,\infty\}$, such a graph is \emph{ $C^k$-structurally stable} if for every $C^\infty$-perturbation 
$(f'_{ij})_{[i,j]\in A}$ of $(f_{ij})_{[i,j]\in A}$, there exists a family of $C^k$ diffeomorphisms $(h_{i})_i\in \prod_{i\in V} Diff^k(M_i, M_i)$ such that the following diagram commutes:
\[\forall [i,j]\in A\; \begin{array}{rcccl}
& &f'_{ij} & &\\ 
&M_i & \rightarrow&M_j &\\
h_i&\uparrow& f_{ij}&\uparrow&h_j\\
&M_i & \rightarrow&M_j & \end{array}\; .\]

The graph $(V, A)$ is \emph{convergent} if for every $[i,j], [i',j']\in A$ if $i=i'$ then $j=j'$. The graph is \emph{without cycle} if for every $n\ge 1$ and every  $([i_k,i_{k+1}])_{0\le k< n}\in V^n$ it holds $i_n\not= i_0$.

\begin{thm}[Mather]
Let $G$ be a graph of smooth proper maps, convergent and without cycle. The graph is $C^\infty$- structurally stable if the following map is surjective:

$(\sigma_i)_i\in \prod_{i\in V}\chi^\infty(M_i)\mapsto (Tf_{ij}\circ\sigma_i-\sigma_j\circ f_{ij})_{[ij]}\in \prod_{[i, j]\in A}\chi^\infty (f_{ij}).$
\end{thm}

Mather gave me an unpublished manuscript of Baas relating his proof, that I developed to study the structural stability of attractor-repellor endomorphisms with possibly a non-empty critical set. 

\begin{defi}  Let $f$ be a smooth endomorphism of a compact, non necessarily connected manifold.  The endomorphism $f$ is {\emph attractor-repellor} if it satisfies  axiom A, and its basic pieces are either expanding pieces or attractors which $f$ acts bijectively. 
\end{defi}

The following theorem generalizes all the results I know (including 
\cite{IPR08} and \cite{IPR10})  about structurally stable maps with non-empty critical set.

\begin{thm}[Berger \cite{Be12}] Let $f$ be an {attractor-repellor}, smooth endomorphism of a compact, non necessarily connected manifold $M$. If the following conditions are satisfied, then $f$ is $C^\infty$-structurally stable:\begin{itemize}
\item[(i)] the singularities $S$ of $f$ have their orbits that do not intersect the non-wandering set $\Omega$,
\item[(ii)] the restriction of $f$ to $M\setminus \hat \Omega$ is $C^\infty$-infinitesimally stable, with $\hat \Omega:= cl\big(\cup_{n\ge 0} f^{-n}(\Omega)\big)$. In other words, the following map is surjective:
\[\sigma \in \Gamma^\infty(M) \mapsto Df\circ \sigma -\sigma \circ f\in 
 \Gamma^\infty(f)\]
 \item[(iii)] $f$ is transverse to the stable manifold of $A$'s points: for any $y\in A$, for any point $z$ in a local stable manifold $W^s_y$ of $y$, for any $n\ge 0$, and for any $x\in f^{-n}(\{z\})$, we have: 
\[Tf^n(T_xM)+T_zW_y^s=T_zM.\]
\end{itemize} 
\end{thm}
Hypothesis $(ii)$ might seem difficult to verify, but it is not. In \cite{Be12} we apply it to many example, even for map for which the critical set does self intersect along its orbit.

It would be intersecting to investigate how the attractor-repeller could be relaxed to enjoy a greater generality. However the $C^0$-equivalently stable singularities are not well classified and so an optimal theorem is today difficult to obtain. Nevertheless, it is not the case in dimension 2. Indeed it is well known that the structurally stable singularities are locally equivalent to one of the following polynomial (called resp. fold and cusp):
\[(x,y)\mapsto (x^2, y)\quad\text{and} \quad(x,y)\mapsto (x^3+xy, y).\]
Hence here is a natural question:
\begin{prob}
Under which hypothesis an axiom A surface endomorphism with singularity is structurally stable?
\end{prob}

\bibliographystyle{alpha}
\bibliography{references}
\end{document}